\documentclass[12pt]{article}
\usepackage[utf8]{inputenc}
\usepackage{amsmath}
\usepackage{amssymb}
\usepackage{amsthm}
\usepackage{fullpage}
\usepackage{graphicx}
\usepackage{hyperref}
\usepackage{color}
\usepackage[sc]{mathpazo}
\usepackage[T1]{fontenc}
\usepackage[margin=10pt,font=small]{caption}

\newtheorem{thm}{Theorem}[section]
\newtheorem{prop}[thm]{Proposition}

\theoremstyle{remark}

\newcommand{\tdef}[1]{\textcolor{blue}{\emph{#1}}}
\newcommand{\horiz}{\begin{center}\rule{0.3\textwidth}{0.5pt}\end{center}}
\newcommand{\id}{\textrm{id}}
\renewcommand{\S}{\mathrm{S}}
\newcommand{\len}{\operatorname{len}}
\newcommand{\lmax}{\operatorname{lmax}}
\newcommand{\rmax}{\operatorname{rmax}}
\newcommand{\slmax}{\operatorname{slmax}}
\newcommand{\asc}{\operatorname{asc}}
\newcommand{\des}{\operatorname{des}}
\newcommand{\srecl}{\operatorname{sldes}}
\newcommand{\size}{\operatorname{size}}
\newcommand{\tails}{\operatorname{tails}}
\newcommand{\lsize}{\operatorname{lsize}}
\newcommand{\rsize}{\operatorname{rsize}}
\newcommand{\fin}{\operatorname{fin}}
\newcommand{\symg}{\mathfrak{S}}

\title{Fighting fish and two-stack sortable permutations}

\author{Wenjie Fang\thanks{\href{mailto:fang@math.tugraz.at}{fang@math.tugraz.at}. Wenjie Fang is currently supported by Austria FWF Grant I2309-N35 and P27290. Part of the present work was done during a postdoc at ENS de Lyon and INRIA in France.} \\ {\small Institute of Discrete Mathematics, Technical University of Graz, Austria}}

\begin{document}

\maketitle

\abstract{In 2017, Duchi, Guerrini, Rinaldi and Schaeffer proposed a new family of combinatorial objects called ``fighting fish'', which are counted by the same formula as more classical objects, such as two-stack sortable permutations and non-separable planar maps. In this article, we explore the bijective aspect of fighting fish by establishing a bijection to two-stack sortable permutations, using a new recursive decomposition of these permutations. With our bijection, we give combinatorial explanations of several results on fighting fish proved previously with generating functions. Using the decomposition of two-stack sortable permutations, we also prove the algebraicity of their generating function, extending a result of Bousquet-Mélou (1998).}

\horiz

\section{Introduction}

In \cite{fighting-fish}, Duchi, Guerrini, Rinaldi and Schaeffer introduced a new class of combinatorial objects called \emph{fighting fish}, which can be seen as a generalization of directed convex polyominoes. They found that the number of fighting fish with $n+1$ lower free edges is given by
\begin{equation} \label{eq:enum}
\frac{2}{(n+1)(3n+1)} \binom{3n+1}{n}.
\end{equation}
This formula also counts various other objects, such as two-stack sortable permutations \cite{west-tsp-def,zeilberger}, non-separable planar maps \cite{tutte-census,brown-nsp}, left ternary trees \cite{left-ternary-tree,nsp-bij-census} and generalized Tamari intervals \cite{tam-def,tam-bij}. In \cite{fighting-fish-enum}, the same authors also proved some refined equi-enumeration results on fighting fish and left ternary trees. However, their proofs used generating functions, thus combinatorially unsatisfactory. The authors then conjectured a still more refined enumerative correspondence between fighting fish and left ternary trees, involving more statistics. They also called for a bijective proof of their conjecture, which is still open to the author's knowledge.

Indeed, unlike the previously mentioned classes of objects, which are linked by a net of bijections, we still lack a combinatorial understanding of fighting fish. The present article is meant to fill this gap by providing a recursive bijection between fighting fish and two-stack sortable permutations. More precisely, our main result is the following (related definitions will be given later).

\begin{thm} \label{thm:main}
  There is a bijection $\phi$ from two-stack sortable permutations to fighting fish satisfying the following conditions. Given a two-stack sortable permutation $\pi$, let $\S(\pi)$ be the result of sorting $\pi$ once through a stack. Suppose that $\pi$ is of length $n$, with $i$ ascents and $j$ descents in $\pi$, and $k$ left-to-right maxima and $\ell$ elements $a$ that precedes $a-1$ in $\S(\pi)$. Then $\phi(\pi)$ is a fighting fish with $n+1$ lower free edges, of which $i+1$ are left and $j+1$ are right, and with fin-length $k+1$ and $\ell+1$ tails.
\end{thm}

This result echoes the conjecture at the end of \cite{fighting-fish-enum}, which calls for a bijection from fighting fish to other objects such as left ternary trees. To prove our result, we first give a new recursive decomposition of two-stack sortable permutations. Then we observe that this new decomposition is isomorphic to a decomposition of fighting fish given in \cite{fighting-fish-enum}, which gives the recursive bijection $\phi$. We finally observe that various statistics are also carried over by $\phi$. Our result can thus be regarded as an equi-enumeration result refined by all related statistics, which can be understood combinatorially. When restricted to a subset of related statistics, we get a combinatorial vision of the refined enumeration results in \cite{fighting-fish-enum}. As a side product, we also prove the algebraicity of a refined generating function of two-stack sortable permutations similar to that in \cite{multistat}, using a simpler functional equation due to the new decomposition.

By providing a bijection between the newly introduced fighting fish and the relatively well-known two-stack sortable permutations, we in fact capture fighting fish in the net of bijections between objects counted by \eqref{eq:enum} we mentioned above. As a result, we could go further in the study of not only fighting fish but also other equi-enumerated objects, such as non-separable planar maps, by looking at structures transferred by our bijection, and natural compositions of our bijection with existing ones.

\section{Preliminaries and previous work}

Given two sequences $A$ and $B$, we denote by $A \cdot B$ their concatenation. The empty sequence (thus also the empty permutation) is denoted by $\epsilon$. We denote by $\len(A)$ the length of a sequence. We now adapt the setting in \cite{multistat}. Let $A = (a_1, a_2, \ldots, a_\ell)$ be a non-empty sequence of distinct integers, with $n$ its largest element. We can write $A$ as $A_L \cdot (n) \cdot A_R$, with $A_L$ (resp. $A_R$) the part of $A$ before (resp. after) $n$. We now define the \tdef{stack-sorting operator}, denoted by $\S$, recursively as
\begin{equation} \label{eq:S-def}
  \S(\epsilon) = \epsilon, \quad \S(A) = \S(A_L) \cdot \S(A_R) \cdot (n).
\end{equation}
For example, $\S(0, -1, 7, 9, 3) = -1,0,7,3,9$ and $\S(6,4,3,2,7,1,5) = 2,3,4,6,1,5,7$.

An equivalent way of thinking about $S$ is that it corresponds to a pass over a ``lazy stack'' $LS$, as described in \cite{west-tsp-def,multistat}. To get $\S(A)$, we start with $LS$ empty, and we push elements of $A$ one by one to $LS$, while maintaining an increasing order of elements in $LS$ from top to bottom. To this end, each time before we push an element $a_i$ into $LS$, we pop every element larger than $a_i$ in $LS$. After all elements are treated, we pop out every element in $LS$, and the overall output sequence is $\S(A)$. 

Given a permutation $\sigma$ in the symmetric group $\symg_n$ viewed as a sequence, we say that $\sigma$ is \tdef{stack-sortable} if $\S(\sigma)$ is the identity $\id_n$ of $\symg_n$. In \cite{taocp}, the following well-known result, expressed using pattern avoidance, was proved by Knuth.
\begin{prop} \label{prop:av231}
  A permutation $\sigma$ is stack-sortable if and only if it avoids the pattern $231$, that is, there are no indices $i<j<k$ such that $\sigma(k) < \sigma(i) < \sigma(j)$.
\end{prop}
We say that $\sigma \in \symg_n$ is a \tdef{two-stack sortable permutation} (or \tdef{2SSP}) if $\S(\S(\sigma)) = \id_n$. We denote by $\mathcal{T}_n$ the set of 2SSPs of length $n$, and $\mathcal{T} = \cup_{n \geq 1} \mathcal{T}_n$ the set of all 2SSPs. We take the convention that the empty permutation $\epsilon$ is not a 2SSP for later compatibility with fighting fish. A characterization of 2SSPs with pattern avoidance can be found \cite{west-tsp-def}.

It was first conjectured by West \cite{west-tsp-def} that the number of 2SSPs of length $n$ is given by \eqref{eq:enum}. Zeilberger provided a proof in \cite{zeilberger} using generating functions. A refined enumeration including various statistics was given by Bousquet-M\'elou in \cite{multistat}. West also observed that \eqref{eq:enum} also counts the number of non-separable planar maps with $n+1$ edges studied by Tutte and Brown \cite{tutte-census,brown-nsp}. A combinatorial proof of West's observation was first given by Dulucq, Gire and Guibert in \cite{eight-bij-refined}, using a sequence of 8 bijections from 2SSPs to a certain family of permutations encoding non-separable planar maps. Then Goulden and West found in \cite{goulden-west} a recursive bijection directly between 2SSPs and non-separable planar maps. They showed that, under specific recursive decompositions, the two classes of objects share the same set of decomposition trees, later called \emph{description trees} in \cite{nsp-bij-census}. Though nice, all these bijections give no direct proof of the enumeration formula. It was in \cite{nsp-bij-census} that Jacquard and Schaeffer finally gave a combinatorial proof of \eqref{eq:enum} by relating description trees to the so-called \emph{left ternary trees}, first studied in \cite{left-ternary-tree}. More recent advances on 2SSPs and related families of permutations defined by sorting through devices like stacks and queues can be found in \cite{bona-survey}.

We now turn to fighting fish defined and studied by Duchi, Guerrini, Rinaldi and Schaeffer in \cite{fighting-fish,fighting-fish-enum}, which can be seen as a generalization of directed convex polyominoes. In the construction, we use \emph{cells}, which are unit squares rotated by 45 degrees. An edge of a cell is \tdef{free} if it is adjacent to only one cell. A \tdef{fighting fish} is constructed by starting with an initial cell called the \emph{head}, then adding cells successively as illustrated on the left side of Figure~\ref{fig:ff}. More precisely, there are three ways to add a new cell (the gray one): (a) we attach it to a free upper right edge of a cell; (b) we attach it to a free lower right edge of a cell; (c) if there is a cell $a$ with two cells $b$ and $c$ attached to its upper right and lower right edge, and such that $b$ (resp. $c$) has a free lower right (resp. upper right) edge, then we attach the new cell to both $b$ and $c$.

\begin{figure}
  \begin{center}
    \includegraphics[scale=1,page=2]{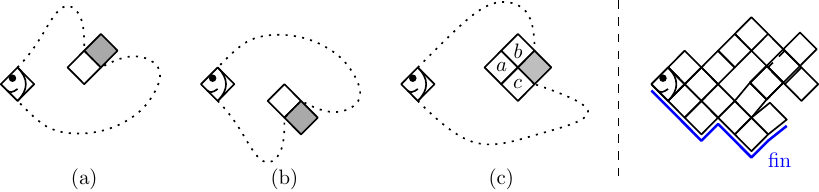}
  \end{center}
  \caption{Adding a cell to a fighting fish, and an example of a fighting fish}
  \label{fig:ff}
\end{figure}

We also need some statistics on fighting fish defined in \cite{fighting-fish-enum}. If a cell has both its right edges free, then its right vertex is called a \tdef{tail}. A fighting fish may have several tails, but it has only one \tdef{nose}, which is the left vertex of its head. The \tdef{fin} of a fighting fish is the path from the nose to the first tail met by following free edges counter-clockwise.

The enumerative properties of fighting fish are studied in \cite{fighting-fish-enum}. It turns out that fighting fish with $n+1$ lower free edges are also counted by \eqref{eq:enum}. Moreover, we have the following refinement.
\begin{prop}[Theorem~2 of \cite{fighting-fish-enum}] \label{prop:enum-refined}
  The number of fighting fish with $i+1$ left lower free edges and $j+1$ right lower free edges is
  \begin{equation}
    \frac{1}{(i+1)(j+1)}\binom{2i+j+1}{j}\binom{i+2j+1}{i}.
  \end{equation}
\end{prop}
Again, this result was proved using generating functions. The same formula was already in \cite{nsp-refined} as the number of non-separable planar maps with $i$ vertices and $j$ faces, and also in \cite{goulden-west,nsp-bij-census} as the number of two-stack sortable permutations with $i$ descents and $j$ ascents. Later we will see a combinatorial explanation via our bijection.

\section{A decomposition of two-stack sortable permutations}

We first lay down some definitions. Given a sequence $A = (a_1, a_2, \ldots, a_n)$ of distinct integers, we define $P(A)$ as the standardization of $a$, that is, the permutation reflecting the order of elements in $a$. For instance, with $A = (0,4,1,9,5,6)$, we have $P(A) = (1,3,2,6,4,5)$. For a permutation $\sigma$, we denote by $\sigma^{+k}$ the sequence obtained by adding $k$ to each element of $\sigma$, and by $\sigma^{+(k_1, m, k_2)}$ with $k_1 < k_2$ the sequence obtained from $\sigma$ by adding $k_1$ to each element strictly smaller than $m$, and adding $k_2$ to other elements. For example, with $\sigma = (6,2,4,1,5,3)$, we have:
\[
\sigma^{+3} = (9,5,7,4,8,6), \quad \sigma^{+(1,3,3)} = (9,3,7,2,8,6).
\]
We observe that, for any permutation $\sigma$ and any values of $k$, $m$ and $k_1 < k_2$, we have $P(\sigma^{+k}) = P(\sigma^{k_1, m, k_2}) = \sigma$. The following statement about $\S$ commuting with these operations is immediate.

\begin{prop} \label{prop:S-commute}
  For any $\sigma \in \symg_n$, we have $\S(\sigma^{+k}) = \S(\sigma)^{+k}$ for any $k \in \mathbb{N}$, and we also have $\S(\sigma^{+(k_1, m, k_2)}) = \S(\sigma)^{+(k_1, m, k_2)}$ for any $0 \leq k_1 < k_2$ and $0 \leq m \leq n$.
\end{prop}
\begin{proof}
  We observe that the operation of $S$ only depends on the relative order of elements, which is the same in $\sigma^{+k}$ and $\sigma^{+(k_1, m, k_2)}$ as in $\sigma$.
\end{proof}

We now present a recursive decomposition of 2SSPs. Let $\pi \in \mathcal{T}_n$ be a 2SSP of size $n$. We suppose that $\pi = \pi_\ell \cdot n \cdot \pi_r$ with $\pi_\ell$ of length $k$. We define $\pi_1 = P(\pi_\ell)$ and $\pi_2 = P(\pi_r)$, and the decomposition is written as $D(\pi) = (\pi_1, \pi_2)$. Here, $\pi_1, \pi_2$ may be empty. The following proposition shows that $D$ is indeed a recursive decomposition.

\begin{prop} \label{prop:valid-decomp}
  For $\pi \in \mathcal{T}_n$ with $n \geq 1$ and $D(\pi) = (\pi_1, \pi_2)$, we have $\pi_1, \pi_2 \in \{\epsilon\} \cup \mathcal{T}$.
\end{prop}
\begin{proof}
  From Proposition~\ref{prop:av231}, we know that $\S(\pi) = \S(\pi_1) \cdot \S(\pi_2) \cdot n$ avoids the pattern $231$, which means that $\S(\pi_1)$ and $\S(\pi_2)$ also avoids $231$. We thus conclude that $\pi_1$ and $\pi_2$ are either empty or in $\mathcal{T}$.
\end{proof}

Now, given $\pi_1, \pi_2$, we exhibit some (in fact, all, \textit{cf.} Propsition~\ref{prop:inverse-2}) possibilities of $\pi \in \mathcal{T}$ such that $D(\pi)=(\pi_1, \pi_2)$, using a new statistic on 2SSPs. Given $\pi \in \mathcal{T}$, we denote by $\slmax(\pi)$ the number of left-to-right maxima in $\S(\pi)$, \textit{i.e.}, the number of indices $i$ such that for all $j<i$ we have $\S(\sigma)(i) > \S(\sigma)(j)$. For example, with $\pi = (3,1,2,5,7,6,4)$, we have $\S(\pi) = (1,2,3,5,4,6,7)$, giving $\slmax(\pi) = 6$. We define $\slmax(\epsilon)=0$. Now suppose that $\pi_1 \in \mathcal{T}_k$ and $\pi_2 \in \mathcal{T}_\ell$. Let $t = \slmax(\pi_2)$, and $a_1, a_2, \ldots, a_t$ be the $t$ left-to-right maxima of $\S(\pi_2)$. We can construct elements in $\mathcal{T}_{k+\ell+1}$ in the following ways:
\begin{itemize}
\item $C_1(\pi_1, \pi_2) = \pi_1 \cdot (k+\ell+1) \cdot \pi_2^{+k}$;
\item $C_2(\pi_1, \pi_2, i) = \pi_1^{+(0,k,a_i)} \cdot (k + \ell + 1) \cdot \pi_2^{+(k-1,a_i+1,k)}$ for $1 \leq i \leq t$.
\end{itemize}
Both constructions are illustrated in Figure~\ref{fig:construction}. In $C_1(\pi_1, \pi_2)$, we allow $\pi_1$ and/or $\pi_2$ to be empty. In $C_2(\pi_1, \pi_2)$, both $\pi_1$ and $\pi_2$ must be non-empty. We now prove that our constructions are valid.

\begin{figure}
  \begin{center}
    \includegraphics[page=1, width=\textwidth]{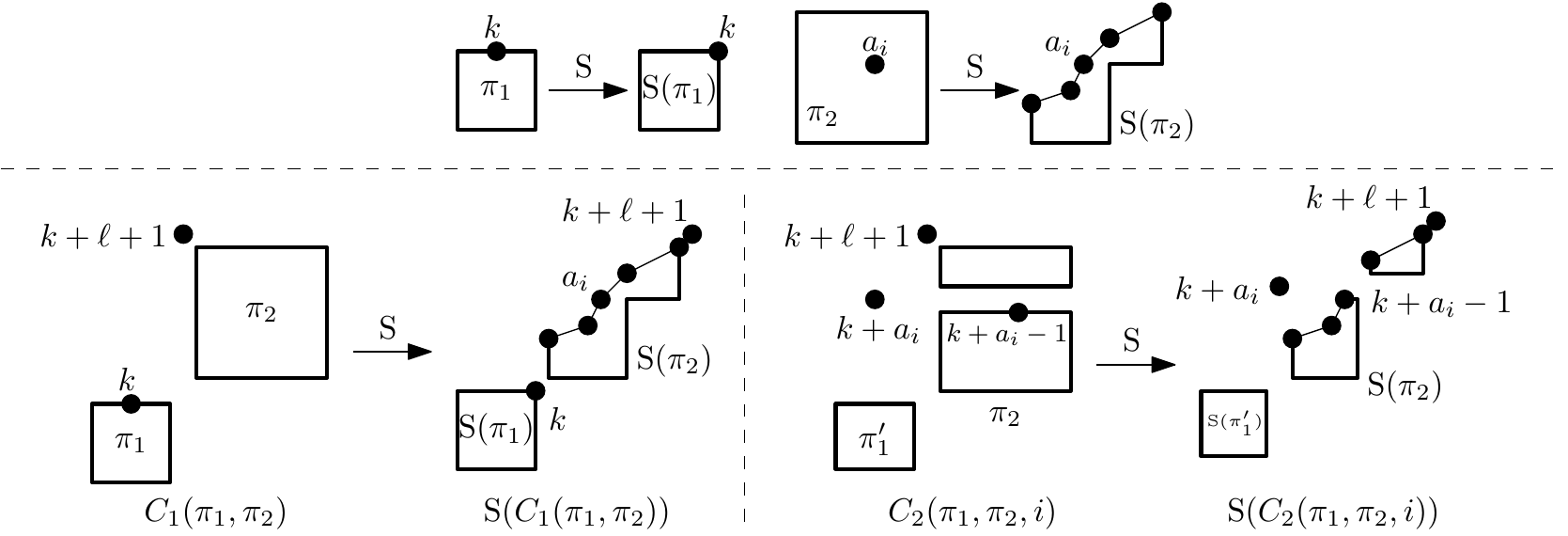}
  \end{center}
  \caption{Constructions $C_1$ and $C_2$}
  \label{fig:construction}
\end{figure}

\begin{prop} \label{prop:c1-valid}
  Given $k, \ell \geq 0$, for any $\pi_1 \in \mathcal{T}_k$ and $\pi_2 \in \mathcal{T}_\ell$, let $\pi = C_1(\pi_1, \pi_2)$. Here we take the convention that $\mathcal{T}_0 = \{ \epsilon \}$. We have $\pi \in \mathcal{T}_{k+\ell+1}$. Furthermore, $\slmax(\pi) = \slmax(\pi_1) + \slmax(\pi_2) + 1$.
\end{prop}
\begin{proof}
  We first observe that $\pi \in \symg_{k+\ell+1}$, since $\pi_1$ covers integers from $1$ to $k$, and $\pi_2^{+k}$ covers integers from $k+1$ to $k+\ell$. With Proposition~\ref{prop:S-commute}, and the fact that $\S(A \cdot B) = \S(A) \cdot \S(B)$ if every element of $A$ is smaller than all elements in $B$, we have
  \begin{align*}
    \S(\pi) &= \S(\pi_1) \cdot \S(\pi_2)^{+k} \cdot (k+\ell+1) \\
    \S(\S(\pi)) &= \S(\S(\pi_1)) \cdot \S(\S(\pi_2))^{+k} \cdot (k+\ell+1) = \id_{k+\ell+1}.
  \end{align*}
In the proof above, since we never specify any element in $\pi_1$ and $\pi_2$, the reasoning also works for $\pi_1$ and/or $\pi_2$ empty.
\end{proof}

\begin{prop} \label{prop:c2-valid}
  Given $k, \ell > 0$, $\pi_1 \in \mathcal{T}_k$ and $\pi_2 \in \mathcal{T}_\ell$, let $t = \slmax(\pi_2)$, and $i$ be an integer between $1$ and $t$. Suppose that $a_i$ is the $i^{\rm th}$ left-to-right maximum of $\S(\pi_2)$. Then we have $\pi = C_2(\pi_1, \pi_2, i) \in \mathcal{T}_{k+\ell+1}$. Furthermore, $\slmax(\pi) = \slmax(\pi_1) +\slmax(\pi_2) - i + 1$.
\end{prop}
\begin{proof}
  We first check that $\pi = \pi_1^{+(0,k,a_i)} \cdot (k+\ell+1) \cdot \pi_2^{k-1, a_i+1, k}$ is in $\symg_{k+\ell+1}$. We see that the set of elements in $\pi_1^{+(0,k,a_i)}$ is $\{ j \mid 1 \leq j \leq k-1 \} \cup \{k+a_i\}$, and that of $\pi_2^{k-1, a_i+1, k}$ is $\{ j \mid k \leq j \leq k+\ell, j \neq k+a_i\}$. We thus know that $\pi$ is indeed in $\symg_{k+\ell+1}$.

  We now check that $\pi$ is in $\mathcal{T}_{k+\ell+1}$. With Proposition~\ref{prop:S-commute}, we have
  \[
  \S(\pi) = \S(\pi_1)^{+(0,k,a_i)} \cdot \S(\pi_2)^{+(k-1,a_i+1,k)} \cdot (k+\ell+1).
  \]
  Now we prove that $\tau = \S(\pi_1)^{+(0,k,a_i)} \cdot \S(\pi_2)^{+(k-1,a_i+1,k)}$ avoids the pattern $231$. Since $\pi_1, \pi_2 \in \mathcal{T}$, both $\S(\pi_1)$ and $\S(\pi_2)$ are stack-sortable, thus avoid $231$, and we only need to prove that there is no pattern $231$ across both parts. By construction, the first part $\S(\pi_1)^{+(0,k,a_i)}$ only has one element $k+a_i$ that is larger than some element in the second part $\S(\pi_2)^{+(k-1,a_i+1,k)}$. Therefore, we only need to check for three elements $b_3 < b_1 < b_2$ with $b_1$ in $\S(\pi_1)^{+(0,k,a_i)}$ and $b_2$ followed by $b_3$ in $\S(\pi_2)^{+(k-1, a_i+1, k)}$. By construction, we must have $b_1 = k+a_i$. But now, since $a_i$ is a left-to-right maximum of $\S(\pi_2)$, the element $b_2$ (thus $b_3$) must occur after $k-1+a_i$ in $\S(\pi_2)^{+(k-1,  a_i+1, k)}$. If such elements $b_2, b_3$ exist, then $k-1+a_i, b_2, b_3$ is a pattern $231$ in $\S(\pi_2)^{+(k-1, a_i+1, k)}$, which is impossible. Therefore, $\tau$ avoids $231$, hence $\S(\pi)$ also, which means $\pi \in \mathcal{T}_{k+\ell+1}$.

  For the equality on $\slmax$, we observe that $\S(\pi_1)^{+(0,k,a_i)}$ contains $k+a_i$, which is larger than the first $i$ left-to-right maxima ($k - 1 + a_j$ for $j \leq i$) in $\S(\pi_2)^{+(k-1, a_i+1, k)}$.
\end{proof}

We now show that the constructions $C_1, C_2$ are the inverse of the decomposition $D$.

\begin{prop} \label{prop:inverse-1}
Given two permutations $\pi_1, \pi_2$ in $\mathcal{T}$, we have $D(C_1(\pi_1, \pi_2)) = (\pi_1, \pi_2)$, and $D(C_2(\pi_1, \pi_2, i)) = (\pi_1, \pi_2)$ for any $1 \leq i \leq \slmax(\pi_2)$.
\end{prop}
\begin{proof}
  It is clear from the constructions of $C_1, C_2$ and $D$, with the fact that, for any permutation $\sigma$, we have $P(\sigma^{+k}) = P(\sigma^{+(k_1, m, k_2)}) = \sigma$.
\end{proof}

\begin{prop} \label{prop:inverse-2}
Let $\pi$ be a permutation in $\mathcal{T}$. Suppose that $D(\pi) = (\pi_1, \pi_2)$. Then either $\pi = C_1(\pi_1, \pi_2)$, or $\pi = C_2(\pi_1, \pi_2, i)$ for some $1 \leq i \leq \slmax(\pi_2)$.
\end{prop}
\begin{proof}
  Let $n$ be the length of $\pi$. We have $\pi = \pi_\ell \cdot n \cdot \pi_r$, and $\S(\pi) = \S(\pi_\ell) \cdot \S(\pi_r) \cdot n$. We also have $\pi_1 = P(\pi_\ell)$ and $\pi_2 = P(\pi_r)$. We now consider elements in $\pi_\ell$ that are larger than the minimum of $\pi_r$. There may be zero, one or more such elements.

  Suppose that no element in $\pi_\ell$ is larger than the minimum of $\pi_r$. In this case, $\pi = C_1(\pi_1, \pi_2)$, and $\pi_1$ and $\pi_2$ can be empty.

  Now suppose that there is exactly one element $m$ in $\pi_\ell$ larger than the minimum of $\pi_r$. In this case, neither $\pi_\ell$ nor $\pi_r$ can be empty. It is clear that $m$ is the largest element in $\pi_\ell$. Let $R_-$ (resp. $R_+$) be the set of elements in $\pi_r$ that are smaller (resp. larger) than $m$. We know that $\S(\pi_\ell)$ ends in $m$, and we write $\S(\pi_\ell)$ as $\tau_1' \cdot m$. We now consider $\S(\pi)$ as
  \[
  \S(\pi) = \tau_1' \cdot m \cdot \S(\pi_r) \cdot n.
  \]
  Since $\S(\pi)$ is stack-sortable, it avoids the pattern $231$. But if an element $r_- \in R_-$ is preceded by an element $r_+ \in R_+$, then $m, r_+, r_-$ is a $231$ pattern. Therefore, we can write $\S(\pi_r) = \tau_r^- \tau_r^+$, where $\tau_r^-$ (resp. $\tau_r^+$) is composed of elements in $R_-$ (resp. $R_+$). The maximum element $m'$ in $\tau_r^-$ must be a left-to-right maximum of $\S(\pi_r)$. Suppose that $m'$ is the $i^{\rm th}$ left-to-right maximum of $\S(\pi_r)$. Since $\S(\pi)$ is a permutation, $m$ is strictly larger than all elements in $\tau_r^-$ and strictly smaller than those in $\tau_r^+$. Therefore, $\S(\pi_r)$ is of the form $\S(\pi_2)^{+(k-1,m'+1,k)}$, where $k$ is the length of $\pi_\ell$. Since $\pi_2 = P(\pi_r)$, we thus have $\pi_r = \pi_2^{+(k-1,m'+1,k)}$, which means $\pi = C_2(\pi_1, \pi_2, i)$.

  In the case where there are at least two elements $m_1, m_2$ in $\pi_\ell$ larger than the minimum $m_3$ of $\pi_r$, we can take $m_2$ the maximum of $\pi_\ell$, and we must have the order $m_1, m_2, m_3$ in $\S(\pi)$, which is an impossible $231$ pattern. We thus conclude the case analysis.
\end{proof}

From the propositions above, under the recursive decomposition $D$, we can build all 2SSPs in a unique way using $\epsilon$ and the constructions $C_1, C_2$. We now study statistics on 2SSPs under these constructions. We first define several statistics, some of which were also studied in \cite{multistat}. Let $\sigma$ be a permutation. We denote by $\lmax(\sigma)$ (resp. $\rmax(\sigma)$) the number of left-to-right (resp. right-to-left) maxima of $\sigma$, \textit{i.e.}, the number of indices $i$ such that for all $j<i$ (resp. $j>i$), we have $\sigma(i) > \sigma(j)$. We also denote by $\asc(\sigma)$ (resp. $\des(\sigma)$) the number of \tdef{ascents} (resp. \tdef{descents}) in $\sigma$, \textit{i.e.}, the number of indices $i$ such that $\sigma(i) < \sigma(i+1)$ (resp. $\sigma(i) > \sigma(i+1)$). Finally, we denote by $\srecl(\sigma)$ the number of \tdef{left descents} in $\S(\sigma)$, \textit{i.e.}, elements $a$ preceding $a-1$ in $\S(\sigma)$. We take the convention that $\lmax(\epsilon) = \rmax(\epsilon) = \asc(\epsilon) = \des(\epsilon) = \srecl(\epsilon) = 0$. We also recall that $\len(\sigma)$ is the length of $\sigma$ as a sequence. The following proposition follows directly from the constructions.

\begin{prop} \label{prop:tsp-stats}
  Given two non-empty permutations $\pi_1, \pi_2$, for any $i$ from $1$ to $\slmax(\pi_2)$, we have
  \begin{align*}
    \lmax(C_1(\pi_1, \pi_2)) = \lmax(C_2(\pi_1, \pi_2, i)) &= \lmax(\pi_1) + 1, \\
    \rmax(C_1(\pi_1, \pi_2)) = \rmax(C_2(\pi_1, \pi_2, i)) &= 1 + \rmax(\pi_2), \\
    \asc(C_1(\pi_1, \pi_2)) = \asc(C_2(\pi_1, \pi_2, i)) &= \asc(\pi_1) + 1 + \asc(\pi_2), \\
    \des(C_1(\pi_1, \pi_2)) = \des(C_2(\pi_1, \pi_2, i)) &= \des(\pi_1) + 1 + \des(\pi_2), \\
    \len(C_1(\pi_1, \pi_2)) = \len(C_2(\pi_1, \pi_2, i)) &= \len(\pi_1) + 1 + \len(\pi_2), \\
    \srecl(C_1(\pi_1, \pi_2)) &= \srecl(\pi_1) + \srecl(\pi_2), \\
    \srecl(C_2(\pi_1, \pi_2, i)) &= \srecl(\pi_1) + \srecl(\pi_2) + 1.
  \end{align*}
  Furthermore, when one of $\pi_1, \pi_2$ is empty, the formulas for $C_1(\pi_1, \pi_2)$ still hold, except that $\asc(C_1(\epsilon, \pi_2)) = \asc(\pi_2)$, and $\des(C_1(\pi_1, \epsilon)) = \des(\pi_1)$.
\end{prop}

\section{Bijection with fighting fish}

In \cite{fighting-fish-enum}, there is a recursive construction of fighting fish called the \emph{wasp-waist decomposition}, which we briefly describe here (and illustrate in Figure~\ref{fig:ff-construct}) for the sake of self-containment. Readers are referred to \cite{fighting-fish-enum} for a detailed definition.

\begin{figure}
  \begin{center}
    \includegraphics[scale=1, page=3]{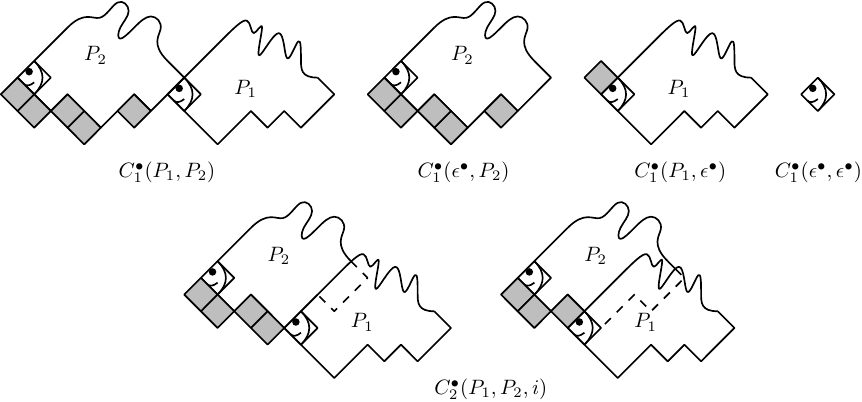}
  \end{center}
  \caption{Constructions $C^\bullet_1$ and $C^\bullet_2$ for fighting fish}
  \label{fig:ff-construct}
\end{figure}

Given two non-empty fighting fish $P_1$ and $P_2$, we build a new fighting fish $C^\bullet_1(P_1, P_2)$ as illustrated in the upper half of Figure~\ref{fig:ff-construct}, by gluing the upper left edge of the head of $P_1$ to the last edge of the fin of $P_2$, then add a new cell to each lower left free edge on the fin. We can also define $C^\bullet_1(P_1, P_2)$ for $P_1, P_2$ being empty (denoted by $\epsilon^\bullet$): $C^\bullet_1(\epsilon^\bullet,P_2)$ is obtained from $P_2$ by adding a new cell to each lower left free edge on the fin; $C^\bullet_1(P_1, \epsilon^\bullet)$ is $P_1$ with a new cell added to the upper left edge of its head; $C^\bullet_1(\epsilon^\bullet, \epsilon^\bullet)$ is the fighting fish with only the head. Now, suppose again that $P_1$ and $P_2$ are non-empty, and $P_2$ has fin-length $k+1$. We observe that $k \geq 1$, since the fin of a fighting fish has length at least $2$. We build $C^\bullet_2(P_1, P_2, i)$ with $1 \leq i \leq k$ as illustrated in the lower half of Figure~\ref{fig:ff-construct}. We first add a new cell to each lower left free edge among the first $k-i+1$ edges on the fin of $P_2$. Then, if the $(k-i+1)$-th edge $e$ is a lower right edge, we glue the head of $P_1$ to $e$, otherwise we glue the head of $P_1$ to the lower right edge of the new cell added to $e$.

It was proved in \cite{fighting-fish-enum} that every fighting fish can be uniquely constructed from $\epsilon^\bullet$ using the constructions $C^\bullet_1, C^\bullet_2$. We now look at some statistics on fighting fish. Given a fighting fish $P$, we denote by $\fin(P)$ the fin-length of $P$, by $\size(P)$ the number of lower free edges in $P$, by $\lsize(P)$ (resp. $\rsize(P)$) the number of left (resp. right) lower free edges in $P$, and by $\tails(P)$ the number of tails in $P$. We take the conventions that $\fin(\epsilon^\bullet)=\lsize(\epsilon^\bullet)=\rsize(\epsilon^\bullet)=\size(\epsilon^\bullet)=\tails(\epsilon^\bullet)=1$. We have the following observation from the definitions of $C^\bullet_1$ and $C^\bullet_2$.

\begin{prop} \label{prop:ff-stats}
  Given two non-empty fighting fish $P_1, P_2$, for any $i$ from $1$ to $\fin(P_2)-1$, we have
  \begin{align*}
    \fin(C^\bullet_1(P_1, P_2)) &= \fin(P_1) + \fin(P_2), \\
    \fin(C^\bullet_2(P_1, P_2, i)) &= \fin(P_1) + \fin(P_2) - i, \\
    \lsize(C^\bullet_1(P_1, P_2)) = \lsize(C^\bullet_2(P_1, P_2, i)) &= \lsize(P_1) + \lsize(P_2), \\
    \rsize(C^\bullet_1(P_1, P_2)) = \rsize(C^\bullet_2(P_1, P_2, i)) &= \rsize(P_1) + \rsize(P_2), \\
    \size(C^\bullet_1(P_1, P_2)) = \size(C^\bullet_2(P_1, P_2, i)) &= \size(P_1) + \size(P_2), \\
    \tails(C^\bullet_1(P_1, P_2)) &= \tails(P_1) - 1 + \tails(P_2), \\
    \tails(C^\bullet_2(P_1, P_2,i)) &= \tails(P_1) + \tails(P_2). \\
  \end{align*}
  Furthermore, the formulas for $C^\bullet_1(P_1, P_2)$ hold for $P_1$ or $P_2$ empty, except that $\lsize(C^\bullet_1(\epsilon^\bullet, P_2)) = \lsize(P_2)$, and $\rsize(C^\bullet_1(P_1, \epsilon^\bullet)) = \rsize(P_1)$.
\end{prop}

Now we define our bijection $\phi$ recursively as follows, using both recursive decompositions of 2SSPs and fighting fish:
\begin{align}
  \begin{split} \label{eq:phi-def}
    \phi(\epsilon) &= \epsilon^\bullet, \\
    \phi(C_1(\pi_1, \pi_2)) &= C_1^\bullet(\phi(\pi_1), \phi(\pi_2)), \\
    \phi(C_2(\pi_1, \pi_2, i)) &= C_2^\bullet(\phi(\pi_1), \phi(\pi_2), i).
  \end{split}
\end{align}

We can now prove our main result.

\begin{proof}[Proof of Theorem~\ref{thm:main}]
  We will prove the conditions of $\phi$ between the set of 2SSPs with $\epsilon$ added and the set of fighting fish with the ``empty fish'' $\epsilon^\bullet$ added.

  We first prove by induction on $\len(\pi)$ that $\phi(\pi)$ is well-defined, with $\slmax(\pi) = \fin(\phi(\pi))-1$. The base case $\pi=\epsilon$ is clear. Now suppose that $\pi$ is not empty, and for every element $\pi' \in \mathcal{T}$ with $\len(\pi') < \len(\pi)$, we have $\phi(\pi')$ well-defined and $\slmax(\pi') = \fin(\phi(\pi))-1$. When $\pi = C_1(\pi_1, \pi_2)$, we see that $\phi(\pi)$ is well-defined. For the case $\pi = C_2(\pi_1, \pi_2, i)$, by induction hypothesis, we have $1 \leq i \leq \slmax(\pi_2) = \fin(\phi(\pi))-1$. Therefore, $\phi(\pi) = C^\bullet_2(\phi(\pi_1), \phi(\pi_2), i)$ is also well-defined. The equality $\slmax(\pi) = \fin(\phi(\pi))-1$ in both cases comes directly from Propositions~\ref{prop:c1-valid},~\ref{prop:c2-valid}~and~\ref{prop:ff-stats}. We thus conclude the induction. We note that, in the case $\pi = C_2(\pi_1, \pi_2, i)$ in the argument above, since $\slmax(\pi') = \fin(\phi(\pi))-1$, every possible value of $i$ in $C^\bullet_2(\phi(\pi_1), \phi(\pi_2), i)$ can be covered by some $\pi$. Therefore, combining with the fact that $C_1, C_2$ (resp. $C_1^\bullet, C_2^\bullet$) give unique construction of 2SSPs (resp. fighting fish), we conclude that $\phi$ is a bijection.

  To prove the correspondences of statistics $\len(\pi) + 1 = \size(\phi(\pi))$, $\asc(\pi) + 1 = \lsize(\phi(\pi))$, $\des(\pi) + 1= \rsize(\phi(\pi))$ and $\srecl(\pi) + 1 = \tails(\phi(\pi))$, we also proceed by induction on the length of $\pi$. We first check that all these agree with the (strange) conventions of 2SSPs and fighting fish. Then we conclude by comparing Proposition~\ref{prop:tsp-stats} against Proposition~\ref{prop:ff-stats}. Details are left to readers.
\end{proof}

Using our bijection, we also recover Proposition~\ref{prop:enum-refined} in a bijective way from known enumeration results on non-separable planar maps with $i$ vertices and $j$ faces in \cite{nsp-refined}. More precisely, these planar maps are sent to 2SSPs with $i$ descents and $j$ ascents by the bijection in \cite{goulden-west}, and then to fighting fish with $i$ right lower free edges and $j$ left lower free edges by our bijection $\phi$. The two statistics can be exchanged with map duality on non-separable planar maps, providing a combinatorial explanation of the symmetry.

\section{Generating function}

We now analyze the generating function of 2SSPs enriched with all the statistics we mentioned before, therefore also that of fighting fish. Let $T(t,x,u,v) \equiv T(t,x,u,v;p,q,s)$ be the generating function defined by
\[
T(t,x,u,v;p,q,s) = \sum_{n \geq 1} \sum_{\pi \in \mathcal{T}_n} t^{n} x^{\slmax(\pi)} u^{\lmax(\pi)} v^{\rmax(\pi)} p^{\asc(\pi)} q^{\des(\pi)} s^{\srecl(\pi)}.
\]
With the symbolic method, from Proposition~\ref{prop:tsp-stats} we have the following equation:
\begin{align}
  \begin{split} \label{eq:multi-stat}
    T(t,x,u,v) &= t x u v (1+qT(t,x,u,1))(1+pT(t,x,1,v)) \\
    &\hspace{3em} + t x u v p q s T(t,x,u,1)\frac{T(t,x,1,v) - T(t,1,1,v)}{x-1}.
  \end{split}
\end{align}

We notice that \eqref{eq:multi-stat} is similar to (2.1) in \cite{fighting-fish-enum}. We have the following result.

\begin{prop} \label{prop:algebraic}
  The generating function $T(t,x,u,v;p,q,s)$ is algebraic in its variables.
\end{prop}
\begin{proof}
  We solve (\ref{eq:multi-stat}) with the quadratic method in a way similar to that in \cite{multistat}, without giving computational details. We denote by $T_{abc}$ with $a,b,c \in \{0,1\}$ the specialization of $T(t,x,u,v)$ where $a=1$ (resp. $b=1$, $c=1$) stands for $x$ (resp. $u$, $v$) specialized to $1$. For instance, $T_{101}$ stands for $T(t,1,u,1)$. We now use this notation to construct the following system for $T$:
  \begin{align}
    T_{000} &= txuv(1+qT_{001})(1+pT_{010}) + txuvpqsT_{001} \frac{T_{010}-T_{110}}{x-1}, \label{eq:000} \\
    T_{010} &= txv(1+qT_{011})(1+pT_{010}) + txvpqsT_{011} \frac{T_{010}-T_{110}}{x-1}, \label{eq:010} \\
    T_{001} &= txu(1+qT_{001})(1+pT_{011}) + txupqsT_{001} \frac{T_{011}-T_{111}}{x-1}, \label{eq:001} \\
    T_{011} &= tx(1+qT_{011})(1+pT_{011}) + txpqsT_{011} \frac{T_{011}-T_{111}}{x-1}, \label{eq:011}
  \end{align}

  Equation~(\ref{eq:011}) is quadratic in $T_{011}$, with the catalytic variable $x$. Therefore, $T_{111}$ and $T_{011}$ is algebraic in related variables (see \cite{BMJ}), and can thus be solved using the quadratic method in particular. Then, Equation~(\ref{eq:001}) is linear in $T_{001}$, and depends further only on the known series $T_{111}$ and $T_{011}$. Therefore, $T_{001}$ is also algebraic in related variables. For Equation~(\ref{eq:010}), it is linear in $T_{010}$, with the catalytic variable $x$, and the equation depends further only on $T_{011}$, which is known to be algebraic. Therefore, $T_{110}$ and $T_{010}$ are both algebraic in all related variables. Finally, from Equation~(\ref{eq:000}) we know that $T_{000}$ is a polynomial in all variables and the algebraic series $T_{001}, T_{010}$ and $T_{110}$, therefore also algebraic itself.
\end{proof}

As a remark, the solution of \eqref{eq:multi-stat} is arguably simpler than that in \cite{multistat}, as there is only one divided difference.

\section{Discussion}

Our bijection $\phi$ is a first step towards a further combinatorial study of fighting fish and two-stack sortable permutations, whose properties are far from being well understood. For instance, flipping along the horizontal axis is an involution on fighting fish. Is this involution related to other involutions in related objects, such as map duality in non-separable planar maps, in a similar way as the case of $\beta$-(1,0) trees and synchronized intervals treated in \cite{duality}? How do all these involutions act on two-stack sortable permutations, which have no apparent involutive structure? We may also ask for recursive decompositions similar to the ones we have studied on other related objects. The conjecture at the end of \cite{fighting-fish-enum} also goes in this direction. As a final question, is there a non-recursive description or variant of the current presented recursive bijection? Such a direct variant would be useful in the structural study of related objects.

\section*{Acknowledgements}

The author thanks Mireille Bousquet-Mélou, Guillaume Chapuy and Gilles Schaeffer for their inspiring discussions and useful comments. The author also thanks the anonymous referees for their precious comments.

\bibliographystyle{alpha}
\bibliography{tsp-ff-fang}

\end{document}